\documentclass[12pt]{amsart}

\usepackage{geometry}
\usepackage{amsmath}
\usepackage{amssymb}
\usepackage{amsthm}
\usepackage{amsfonts, dsfont}
\usepackage{mathrsfs} 
\usepackage{paralist}
\usepackage{graphics} 
\usepackage{epsfig} 
\usepackage{graphicx}  
\usepackage{epstopdf}
\usepackage{epstopdf}
\usepackage{verbatim}
\usepackage{mathrsfs}
\usepackage{mathtools}
\usepackage{pstricks}
\usepackage{relsize}
\usepackage{tikz}
\usetikzlibrary{matrix}
\usepackage{subcaption}
\usepackage{pgfplots}
\usepackage{fixltx2e}
\usepackage{enumitem}
\usepackage{bm}
\usepackage{upgreek}
\usepackage{url}
\usepackage[colorlinks=true]{hyperref}
\hypersetup{urlcolor=blue, linkcolor=blue, citecolor=red}
\usepackage{cleveref}
\usepackage{bm}
\usepackage{appendix}

\usepackage{algorithm2e}

\allowdisplaybreaks

\newcommand{\dt}{{\rm d} t}
\newcommand{\ds}{{\rm d} s}
\newcommand{\dx}{{\rm d} x}
\newcommand{\rd}{{\rm d} }
\newcommand{\RR}{{\mathbb R}}
\newcommand{\EE}{{\mathbb E}}
\newcommand{\mE}{{\mathcal E}}

\newcommand{\hY}{{\widehat{Y}}}

\newcommand{\mc}[1]{\mathcal{#1}}

\newcommand{\la}{\langle}
\newcommand{\ra}{\rangle}

\DeclareMathOperator*{\argmin}{argmin}

\usepackage[abbrev]{amsrefs}

\usepackage{amssymb}

\newtheorem{theorem}{Theorem}[section]
\newtheorem{lemma}[theorem]{Lemma}
\newtheorem{assum}[theorem]{Assumption}
\newtheorem{corollary}[theorem]{Corollary}
\newtheorem{proposition}[theorem]{Proposition}
\newtheorem{remark}[theorem]{Remark}

\definecolor{ForestGreen}{RGB}{34,139,34}
\definecolor{ao(english)}{rgb}{0.0, 0.5, 0.0}

\begin{document}

\title[Self-interacting CBO]{Self-interacting CBO: Existence, uniqueness, and long-time convergence}
\thanks{
}

\author{Hui Huang}
\author{Hicham Kouhkouh}
\address{Hui Huang and Hicham Kouhkouh \newline \indent
University of Graz\newline \indent
Department of Mathematics and Scientific Computing, NAWI, Graz, Austria
}
\email{\texttt{hui.huang@uni-graz.at}, \texttt{hicham.kouhkouh@uni-graz.at}}
\thanks{}






\date{\today}

\begin{abstract}
A self-interacting dynamics that mimics the standard Consensus-Based Optimization (CBO) model is introduced. This single-particle dynamics is shown to converge to a unique invariant measure that approximates the global minimum of a given function. As an application, its connection to CBO with Personal Best introduced by C. Totzeck and M.-T. Wolfram (Math. Biosci. Eng., 2020) has been established.
\end{abstract}












\maketitle

\section{Introduction}

Let $(\Omega,\mathfrak{F},\mathfrak{F}_t,\mathbb P)$ be a complete filtered probability space, and let $(B_t)_{t\geq 0}$ be a standard $d$-dimensional Brownian motion defined therein. We consider a McKean-Vlasov process $X_{\cdot}$ on $\RR^d$ whose dynamics is the one of the consensus-based optimization (CBO) model \cite{pinnau2017consensus}, and is governed by the following equation
\begin{equation}\label{CBO}
\begin{aligned}
    \rd X_t & =-\lambda(X_t-\mathfrak{m}_{\alpha}(\mathcal{L}_{t}[X])) \, \dt + \sigma D(X_t-\mathfrak{m}_{\alpha}(\mathcal{L}_{t}[X])) \, \rd B_t\,,
\end{aligned}
\end{equation}
where $\mathcal{L}_{t}[X] := \text{Law}(X_{t})$, and $\sigma>0$ is a real constant. Here, we employ anisotropic diffusion in the sense that $D(X):=\mbox{diag}(|x_1|,\dots,|x_d|)$ for any $X=(x_{1},\dots, x_{d})\in\RR^d$, which has been proven to handle high-dimensional problems more effectively \cite{carrillo2021consensus,fornasier2022anisotropic}.
The \textit{current global consensus} point $\mathfrak{m}_{\alpha}(\mathcal{L}_{t}[X])$ is defined by
 \begin{equation}\label{XaN}
    \mathfrak{m}_{\alpha}(\mathcal{L}_{t}[X]) = \frac{\int_{\RR^d} x \, \omega_{\alpha}^{f}(x)\; \mathcal{L}_{t}[X](\dx)}{\int_{\RR^d}\omega_{\alpha}^{f}(x)\; \mathcal{L}_{t}[X](\dx)}\,,
\end{equation}
and the weight function is chosen to be \; $\omega_\alpha^f(x):=\exp(-\alpha f(x))$. \; 
This choice of weight function is motivated by the well-known Laplace's principle \cite{miller2006applied,MR2571413}.
The CBO method is proposed to solve the optimization problem:
\begin{equation*}
	\text{Find } \; x^* \in \argmin_{x\in \RR^d} f(x),\,
\end{equation*}
where $f$ can be a non-convex non-smooth objective function that one wishes to minimize.

Under certain assumptions on the well-prepared initial data and parameters, it can be proved as in \cite{carrillo2018analytical} that, for any fixed $\alpha>0$ , $\rho_t$ converges to a Dirac measure. If moreover $\alpha$ is chosen large enough, it can also be shown that the latter Dirac measure can be supported on a point close to $x^*$, a global minimizer of $f$. 

A common way to run the CBO dynamics \eqref{CBO} is through the large particle limit (or mean-field limit). Namely, one considers the following $N$ interacting particle system $\{X_{\cdot}^i\}_{i=1}^N$ satisfying
\begin{equation}\label{CBOparticle}
\rd X_t^i=-\lambda(X_t^i - \mathfrak{m}_\alpha(\rho_t^{N}))\dt+\sigma D(X_t^i-\mathfrak{m}_\alpha(\rho_t^{N}))\rd B_t^{i},\quad i=1,\dots, N\,,
\end{equation}
where $\rho_t^N=\frac{1}{N}\sum_{i=1}^N\delta_{X_t^i}$ is the empirical measure associated to the particle system, and $\{B_.^i\}_{i=1}^N$ are $N$ independent $d$-dimensional Brownian motions. The convergence from the particle system \eqref{CBOparticle} to the McKean-Vlasov process \eqref{CBO} \cite{fornasier2020consensus,huang2022mean,gerber2023mean} is called the mean-field limit, hence the name \textit{mean-field interaction}.

In the long-time limit and for well-chosen parameters (in particular for $\alpha\gg 1$), this system of (finitely many) particles would converge to a single point which is the global minimizer of $f$, and it coincides with the support of the long-time limit of $\mathcal{L}_{t}[X]$, the law of the McKean-Vlasov process defined above. In other words, this suggests that a Dirac measure supported on the global minimizer is an invariant measure for \eqref{CBO}. By definition, a probability measure $\mu$ is \underline{invariant} for $X_{\cdot}$ if and only if it is a fixed point for the adjoint of its (nonlinear) transition semigroup $\{T_{t}\}_{t\geq 0}$, that is $T^{*}_{t}\mu = \mu$ for all $t> 0$. The adjoint operator $T^{*}_{t}$ is defined on the space of probability measures as $T^{*}_{t}\nu := \mathcal{L}_{t}[X]$ when $\nu=\mathcal{L}_{0}[X]$. 
Using the generator $A$ of the semigroup $T_{t}$ we have the characterization: a probability measure $\mu$ is invariant for  $\{T_{t}\}_{t\geq 0}$ if and only if
\begin{equation*}
    \int_{\mathbb{R}^{d}} A\varphi(x)\,\text{d}\mu(x) = 0\quad \forall \, \varphi \text{ in the domain of } A.
\end{equation*}
This is the distributional definition of $\mu$ being a solution to the Kolmogorov-Fokker-Planck equation $A^{*}\mu = 0$, where $A^{*}$ is the adjoint operator of the generator $A$. For a McKean-Vlasov process, $A$ is a nonlinear operator.

For the process \eqref{CBO}, it can be easily verified that  any Dirac measure (not necessarily supported on the global minimizer) is invariant. This non-uniqueness is mainly due to two reasons: in the drift, the term $-X$ is not strong enough with respect to the mean-field term $\mathfrak{m}_{\alpha}(\cdot)$, and the diffusion degenerates. 
With these features, it becomes difficult to study the dynamical properties of \eqref{CBO}, in particular its long-time behavior. 
To remedy to this situation, we propose a modification of \eqref{CBO} which addresses exactly these two mentioned issues. Hence, we consider a \textit{rescaled} CBO given by
\begin{equation}\label{CBO kappa intro}
    \rd X_t =-\lambda(X_t - \kappa\,\mathfrak{m}_{\alpha}(\mathcal{L}_{t}[X]))\dt+\sigma\left(\frac{1}{\alpha}\mathds{I}_{d} + D(X_t - \kappa\, \mathfrak{m}_{\alpha}(\mathcal{L}_{t}[X])) \right)\rd B_t
\end{equation}
complemented with a deterministic initial condition $X_{0}=x\in \mathbb{R}^{d}$,  where $\mathds{I}_{d}$ is the $d$-dimensional identity matrix, $0<\kappa<1$ is a small positive constant, and $\alpha$ is the same as in the consensus point\footnote{In fact we could have chosen any positive constant $\delta>0$, and write the diffusion matrix as $\delta\,\mathds{I}_{d} + \sigma\,D(X_t - \kappa \,\mathfrak{m}_\alpha(\mathcal{L}_{t}[X]))$.}. 
As we shall later see, 
$\EE[X_{t=\infty}]\approx \kappa\, x^*$ for sufficiently large $\alpha\gg 1$.  Clearly, the  approximation using a system of $N$ interacting particles is still valid for \eqref{CBO kappa intro}, and it is of the form
 \begin{equation}\label{CBOkappa particle}
\rd X_t^i =-\lambda \left((X_t^i - \kappa\,\mathfrak{m}_{\alpha}(\rho_t^N) \right) \dt+\sigma\left(\frac{1}{\alpha}\mathds{I}_{d} + D(X_t^i - \kappa\, \mathfrak{m}_{\alpha}(\rho_t^N) \right)\rd B_t^i,\quad i=1,\dots, N\,.
\end{equation}
The dynamics \eqref{CBOkappa particle} has recently also been used in \cite{herty2024multiscale}, although in a different context.  

Our \textbf{main contribution}  is the approximation of the invariant measure of \eqref{CBO kappa intro} using only one \textit{self-interacting} particle, rather than the conventional $N$-particle system \eqref{CBOkappa particle}. We will also show that the invariant measure exists and is moreover unique. A polynomial rate of convergence towards it is also proven.


\section{The model of self-interacting CBO}

We introduce the following self-interacting CBO model with respect to the rescaled CBO \eqref{CBO kappa intro}
\begin{equation}\label{CBOself'}
    \rd Y_t = - \lambda \big( Y_t - \kappa \, \mathfrak{m}_\alpha(\mE_t[Y])\big)\dt+\sigma\left(\frac{1}{\alpha}\mathds{I}_{d} + D(Y_t - \kappa \, \mathfrak{m}_\alpha(\mE_t[Y]))\right) \rd B_t\,,
\end{equation}
complemented with the deterministic\footnote{One particle only substitutes for the the $N$-interacting particles. So we do not need a randomly distributed initial condition.} initial condition $Y_{0}=x\in \mathbb{R}^{d}$, and where 
\begin{equation*}
    \mE_{t} [Y]:=\frac{1}{t}\int_0^t\delta_{Y_r}\,\rd r=\int_0^1\delta_{Y_{ts}}\,\rd s,\quad t>0
\end{equation*}
is usually referred to as the \textit{occupation measure}. Indeed, given an open set $\mathcal{O}\subset \mathbb{R}^{d}$, we have $\mathcal{E}_{t}[X](\mathcal{O}) = \frac{1}{t}|\{ r\in [0,t]\,:\, X_{r}\in \mathcal{O}\}|$, where $|I|$ denotes Lebesgue measure of an interval $I$. So $\mathcal{E}_{t}[X](\mathcal{O})$ measures the average of the time during which the process $X$ \textit{occupies} an open set $\mathcal{O}$ up to time $t$. \\
With the definition of the current consensus point \eqref{XaN}, we now have
\begin{equation*}
    \mathfrak{m}_\alpha(\mE_t[Y]) = \frac{\int_0^tY_r\exp(-\alpha f(Y_r))\rd r}{\int_0^t\exp(-\alpha f(Y_r))\rd r}=\frac{\int_0^1Y_{st}\exp(-\alpha f(Y_{st}))\rd s}{\int_0^1\exp(-\alpha f(Y_{st}))\rd s}\,.
\end{equation*}
The above term $\mathfrak{m}_\alpha(\mE_t[Y])$ has been used in \cite[Equation (2.4)]{totzeck2020consensus}, albeit with a different motivation. There, it is referred to as the ``Personal Best'' term, which approximates the best location along the entire trajectory of the particle $Y$ up to time $t$ when $\alpha$ is sufficiently large. Comparing \eqref{CBOself'} to our rescaled CBO model
\begin{equation}\label{CBO kappa}
    \rd X_t  =-\lambda\big(X_t - \kappa\, \mathfrak{m}_\alpha(\mathcal{L}_t[X])\big)\dt+\sigma\left(\frac{1}{\alpha}\mathds{I}_{d} + D(X_t - \kappa \, \mathfrak{m}_\alpha(\mathcal{L}_t[X])) \right)\rd B_t\,,
\end{equation}
the difference is in $\mathfrak{m}_{\alpha}(\cdot)$: in the first it is evaluated in $\mE_t[Y]$, while in the second we have $\mathcal{L}_t[X]=\text{Law}(X_{t})$. 

Remarkably, from a practical standpoint, simulating the occupation measure of the self-interacting process \eqref{CBOself'} requires only one single particle. This feature distinctly sets it apart from the conventional $N$-particle approximation \eqref{CBOkappa particle}, where its global convergence usually requires an additional condition that the number of particles $N$ goes to infinity \cite{fornasier2021consensus,fornasier2021consensus1}. 
Moreover, a uniform-in-time mean-field limit for particle system \eqref{CBOkappa particle} has recently been established in \cite{huang2024uniform}, whereas for the particle system \eqref{CBOparticle}, this remains an open problem. 
On the other hand, such self-interacting process has recently demonstrated its applicability in training two-layer neural networks 
\cite{du2023self}, and used in games and control \cite{kouhkouh24one}.

\hfill

\begin{assum}\label{assum1}
	We assume the following properties for the objective function.
	\begin{enumerate}
		\item $f:~\RR^d\to \RR$ is bounded from below by $\underline f=\min f$, and there exist $L_{f}>0, s\geq 0$ constants such that
		\begin{equation*}
			|f(x)-f(y)|\leq L_f(1+|x|+|y|)^s|x-y| \quad \forall x,y\in \RR^d\,.
		\end{equation*}
		\item There exist constants $c_1,c_2>0$  and $\ell>0$ such that
		\begin{equation*}
			c_1(|x|^\ell-1)	\leq f-\underline f\leq c_2(|x|^\ell+1)\quad \forall x\in \RR^d.
		\end{equation*}
	\end{enumerate}
\end{assum}

\begin{lemma}\label{lem: wellposed}  
Let $f$ satisfy Assumption \ref{assum1}, then both the dynamics \eqref{CBOself'} and \eqref{CBO kappa} have strong solutions. 
\end{lemma} 

\begin{proof}
The well-posedness of the SDE \eqref{CBOself'} follows the same arguments presented in \cite[Theorem 2]{totzeck2020consensus}, demonstrating that both the drift and diffusion terms in \eqref{CBOself'} satisfy the local Lipschitz and linear growth conditions. The well-posedness of the SDE \eqref{CBO kappa} can be verified similarly as in \cite[Theorem 3.2]{carrillo2018analytical} or \cite[Theorem 2.4]{gerber2023mean} using Leray-Schauder fixed point theorem. 
\end{proof}

In what follows, $\|\cdot\|$ denotes the Frobenius norm of a matrix and $|\cdot|$ is the standard Euclidean norm in $\RR^d$; $\mathscr{P}(\RR^d)$ denotes the space of probability measures on $\RR^d$, and $\mathscr{P}_p(\RR^d)$ with $p\geq 1$ contains all $\mu\in \mathscr{P}(\RR^d)$ such that $\mu(|\cdot|^p):=\int_{\RR^d}|x|^p\mu(\dx)<\infty$; it is equipped with $p$-Wasserstein distance $W_p(\cdot,\cdot)$. 
Lastly, we define $\mathscr{P}_{p,R}(\mathbb{R}^{d}) := \{\mu \in \mathscr{P}_{p}(\mathbb{R}^{d})\,:\, \mu(|\cdot|^{p})\leq R\}$ where $R>0$ is a constant. 

One can further generalize $\mE_{t}[\cdot]$ and $\mathcal{L}_{t}[\cdot]$ to their weighted forms
\begin{equation}\label{eq:weighted measures}
\begin{aligned}
    & \mE_t^\pi[Y] := \int_0^1 \delta_{Y_{st}}\,\pi_t(\text{d}s)\quad \text{ and } \quad \mathcal{L}^{\varpi}_{t}[X] := \int_{0}^{1} \mathcal{L}_{st}[X]\,\varpi_{t}(\text{d}s)\\
    & \quad \quad \text{ for } \quad \pi,\varpi\in \Pi :=\left\{\pi=(\pi_t)_{t\geq 0}:~\pi_t\in \mathscr{P}([0,1])\right\}\,,
\end{aligned}
\end{equation}
where $\mathscr{P}([0,1])$ is the space of probability measures on $[0,1]$.
Then we have the generalized self-interacting CBO for some $\pi \in \Pi$ 
\begin{equation}\label{CBOself pi}
	\rd Y_t = - \lambda ( Y_t -\kappa \,\mathfrak{m}_\alpha(\mE_t^\pi[Y])) \dt + \sigma \left(\frac{1}{\alpha}\mathds{I}_{d} + D(Y_t - \kappa\, \mathfrak{m}_\alpha(\mE_t^\pi[Y])) \right) \rd B_t\,, \;Y_{0}=x\in \mathbb{R}^{d}
\end{equation}
and the generalized \textit{rescaled} mean-field CBO for some $\varpi\in \Pi$ 
\begin{equation}\label{CBO pi}
	\rd X_t = - \lambda ( X_t -\kappa \,\mathfrak{m}_\alpha(\mathcal{L}_t^\varpi[X])) \dt + \sigma \left(\frac{1}{\alpha}\mathds{I}_{d} + D(X_t - \kappa\, \mathfrak{m}_\alpha(\mathcal{L}_t^\varpi[X])) \right) \rd B_t\,,\; X_{0}=x\in \mathbb{R}^{d}.
\end{equation}
In the latter two cases, we have 
\begin{equation*}
	\mathfrak{m}_\alpha(\mathcal{L}_t^\varpi[X])) = 
	\frac{ 
	\int_{0}^{1}\int_{\RR^d} x \, \omega_{\alpha}^{f}(x)\mathcal{L}_{st}[X](\dx) \, \varpi(\text{d}s)
	    }{
	\int_{0}^{1}\int_{\RR^d}\omega_{\alpha}^{f}(x)\mathcal{L}_{st}[X](\dx)\,\varpi(\text{d}s)
	    }, \:
	\text{ and } \;
	\mathfrak{m}_\alpha(\mathcal{E}_t^\pi[Y])) = 
	    \frac{ 
	    \int_{0}^{1}\int_{\RR^d} x \, \omega_{\alpha}^{f}(x)\delta_{Y_{st}}(\dx)\pi_{t}(\text{d}s) 
	    }{
	    \int_{0}^{1}\int_{\RR^d}\omega_{\alpha}^{f}(x)\delta_{Y_{st}}(\dx) \pi_{t}(\text{d}s)
    }.
\end{equation*}
It is obvious that $\mE_t[\cdot]$ can be represented as $\mE_t^\pi[\cdot]$ with $\pi_t$ being the Lebesgue measure, while  $\mathcal{L}_{t}[\cdot]$ corresponds to choosing $\varpi\equiv \delta_{1}$, a Dirac measure, in  $\mathcal{L}_{t}^{\varpi}[\cdot]$.

\section{Equivalence proof}
 
Our main goal in the sequel is to establish a level of equivalence between mean-field interacting model \eqref{CBO kappa} and self-interacting model \eqref{CBOself'}, in the sense that they both evolve eventually towards the same equilibrium state.  This allows us to run CBO algorithm of \eqref{CBOself'} instead, and as an alternative to the $N$-interacting particles \eqref{CBOkappa particle}. 

First, let us recall some estimates on $\mathfrak{m}_{\alpha}(\mu)$ from \cite[Corollary 3.3, Proposition A.3]{gerber2023mean}.
\begin{lemma}\label{lem: useful estimates}
Suppose that $f$ satisfies Assumption \ref{assum1}. Then for all $R>0$, 
there exists some constant $L_{\mathfrak{m}}>0$ depending on $R$ such that
\begin{equation}
    |\mathfrak{m}_{\alpha}(\mu)-\mathfrak{m}_{\alpha}(\nu)|\leq L_{\mathfrak{m}} W_2(\mu,\nu)\quad \forall (\mu,\nu)\in \mathscr{P}_{2,R}(\RR^d)\times \mathscr{P}_{2}(\RR^d)\,.
\end{equation}
Moreover for all $q\geq 1$, there exists constant $C_1>0$ depending on $q,c_1,c_2,\ell$ such  that
\begin{equation}
    |\mathfrak{m}_{\alpha}(\nu)|\leq C_1\left(\int_{\RR^d}|x|^q\nu(dx)\right)^{\frac{1}{q}}\quad \forall \nu\in \mathscr{P}_q(\RR^d)\,.
\end{equation}
\end{lemma}


Let us introduce the following notation for $(x,\mu)\in \mathbb{R}^{d}\times \mathscr{P}(\mathbb{R}^{d})$
\begin{equation*}
	b(x,\mu):=-\lambda(x-\kappa\,\mathfrak{m}_{\alpha}(\mu)),\quad \text{ and }\quad \bm{\sigma}(x,\mu):=\left(\frac{1}{\alpha}\mathds{I}_{d} + \sigma D(x-\kappa\,\mathfrak{m}_{\alpha}(\mu))\right),
\end{equation*}
where $\lambda, \sigma,\kappa,\alpha$ are positive constants. Then it is easy to check that the following estimates hold (see the supplementary material \ref{app: lem}  in the appendix).

\begin{lemma}\label{lemsta}
For all $x,y\in\RR^d$ and $(\mu,\nu)\in \mathscr{P}_{2,R}(\RR^d)\times \mathscr{P}_{2}(\RR^d)$ it holds that
\begin{equation}\label{dissip}
    2 \la b(x,\mu)-	b(y,\nu),x-y\ra +\|\bm\sigma(x,\mu)-\bm\sigma(y,\nu)\|^2 
    \leq -\mathfrak{a}|x-y|^2+ \mathfrak{b} W_2^2(\mu,\nu)\,.
\end{equation}
Also, there exists $K>0$ such that for all $x\in\RR^d, \nu\in \mathscr{P}_2(\RR^d), \delta>0$, it holds that
\begin{align}\label{lemeq}
    2\la b(x,\nu),x\ra +(1+\delta)\|\bm\sigma(x,\nu)\|^2
    \leq -\mathfrak{c}|x|^2+ K(1 + \delta + \nu(|\cdot|^{2})) \,.
\end{align}
When  $\lambda > 8\sigma^{2}$, $\delta=1$, and for $0<\kappa \ll 1$, the above inequalities are satisfied with  $\mathfrak{a} >2 \mathfrak{b} \geq 0$, and $\mathfrak{c}>0$.
\end{lemma}

\begin{remark}\label{rmk: coeff}
    In Lemma \ref{lemsta}, the constants of the first inequality are   $\mathfrak{a} = 2\lambda-\lambda \kappa - 2\sigma^{2}$ and $\mathfrak{b} = (\lambda + 2\,\sigma^{2}\,\kappa)\,\kappa\, L_{\mathfrak{m}}^{2}$, and in the second inequality, we have for any $\delta> 0$, the constants $\mathfrak{c} = (2\lambda - \lambda \kappa ) - 2\sigma^{2}(1+\delta)(1+\kappa)$, and $K= \max\left\{ 2\sigma^{2}d/\alpha^{2} ,\; \big[\lambda \kappa \,C_{1}^{2} + 2\,\sigma^{2}(1+\delta)(1+\kappa)\,\kappa\, C_{1}^{2}\big] \right\}$. 
    See the supplementary material \ref{app: lem} in the appendix.
\end{remark}

Observe that the assumption of having at least one of the measures chosen within  $\mathscr{P}_{2,R}(\mathbb{R}^{d})$ is used to guarantee the validity of the estimates in Lemma \ref{lem: useful estimates}. 

Here and in what follows, we shall call \underline{the data of the problem} the parameters $\lambda, \sigma, \kappa, \alpha$, together with the constants in Assumption \ref{assum1}, in Lemma \ref{lem: useful estimates}, and in Lemma \ref{lemsta}.

\begin{proposition}\label{prop:existence of inv meas}
The dynamics \eqref{CBO kappa} has an invariant measure $\mu_\alpha^*$ in $\mathscr{P}_{2,R}(\mathbb{R}^{d})$, and $R$ only depends on the data of the problem.
\end{proposition} 

\begin{proof}
Using Lemma \ref{lem: useful estimates} and Lemma \ref{lemsta}, the existence of an 
invariant measure becomes a consequence of \cite[Theorem 2.2]{zhang2023existence} where $R$ is defined in the latter's proof.  More details are given in the supplementary material \ref{app: prop} in the appendix. 
\end{proof}

With the invariant measure $\mu_{*}^{\alpha}$ of \eqref{CBO kappa} in hand, we can write a formal global convergence result. Indeed, let 
$\mu_{*}^{\alpha}=\mathcal{L}_{\infty}[X]$, then the following holds
\begin{equation}
\label{eq:formal conv}
    0=\frac{\rd \EE[X_\infty]}{\dt}=-\lambda (\EE[X_\infty]-\kappa\,\mathfrak{m}_{\alpha}(\mathcal{L}_{\infty}[X])) \; \Rightarrow \;  \EE[X_\infty] 
    = \kappa \,\mathfrak{m}_{\alpha}(\mu_\alpha^*)=\kappa\int_{ \RR^d }x\,\eta_\alpha^*(\dx)\,,
\end{equation}
where $\eta_\alpha^*(\dx):=\frac{\omega_\alpha^f(x)\mu_\alpha^*(\dx)}{\int_{ \RR^d }\omega_\alpha^f(x)\mu_\alpha^*(\dx)}$.
If additionally we assume for any $\epsilon>0$, there exists $C_\epsilon>0$ a constant independent of $\alpha$ such that $\mu_\alpha^*(\mc A_\epsilon)\geq C_\epsilon$ with
$\mc A_\epsilon:=\left\{x\in\RR^d:~e^{-f(x)}>e^{-f(x^*)}-\epsilon \right\}$ where $x^{*}$ is a global minimizer, then according to  \cite[Lemma A.3]{huang2024consensus} it holds that
\begin{equation}\label{lap_princ}
	\lim\limits_{\alpha\to\infty}\left(-\frac{1}{\alpha}\log\left(\int_{ \RR^d }\omega_\alpha^f(x)\mu_\alpha^*(\dx)\right)\right)= f(x^*)\,.
\end{equation}
Denoting integration by $\langle \,\cdot,\,\cdot\,\rangle$, the latter means that for the indicator function $\textbf{I}_{\{x^*\}}(\cdot)$, it holds 
\begin{equation*}
	\lim\limits_{\alpha\to\infty} \la \frac{\omega_\alpha^f(x)\mu_\alpha^*(\dx)}{\int_{ \RR^d }\omega_\alpha^f(x)\mu_\alpha^*(\dx)},\textbf{I}_{\{x^*\}}\ra=\lim\limits_{\alpha\to\infty}\la \eta_\alpha^*(\dx),\textbf{I}_{\{x^*\}}\ra=1\,.
\end{equation*}
Thus $\eta_\alpha^*$ approximates the Dirac distribution $\delta_{x^*}$ for large $\alpha\gg 1$. See some additional comments in the supplementary material \ref{app: comp} in the end of appendix.  Consequently, $\mathfrak{m}_{\alpha}(\mu_\alpha^*)$ provides a good estimate of $x^*$, which (recalling \eqref{eq:formal conv}) leads to the fact that $\EE[X_\infty]\approx \kappa\, x^*$ for sufficiently large $\alpha\gg 1$. To rigorously prove the convergence, it is necessary to verify the assumption that for any $\epsilon > 0$, there exists a constant $C_\epsilon > 0$ independent of $\alpha$ such that $\mu_\alpha^*(\mathcal{A}_\epsilon) \geq C_\epsilon$. This is an undergoing work.

Although formal, these computations can already exhibit some advantages of our model regarding the standard CBO model. Indeed, as we have mentioned  in the introduction, any Dirac measure is invariant for the standard CBO model. Whereas in our case, the use of the parameter $\kappa$ in the dynamics makes all Dirac measures not invariant. This makes possible the use of the Laplace principle in \eqref{lap_princ}. Intuitively, this creates a disconnection between the limit in the time variable ($t\to +\infty$), and the limit $\alpha\to +\infty$, hence allowing to easily handle the convergence of the dynamics towards the (to-be) global minimum. \\
Shortly after, we will show that this invariant measure $\mu_{*}^{\alpha}$ is in fact unique, which consolidates this intuition.

We shall use the classes of weight families introduced in \cite{du2023empirical}. Recalling the set $\Pi$ in \eqref{eq:weighted measures} of flows of probabilities in $[0,1]$, we define:
\begin{equation*}
\begin{aligned}
	\Pi_{1}(\varepsilon) & = \left\{ \pi \in \Pi \,:\, \limsup\limits_{t\to +\infty} \int_{0}^{1} s^{-\varepsilon} \, \pi_{t}(\text{d}s)\, < \, \frac{\mathfrak{a}}{\mathfrak{b}} \right\}, \quad \text{ where } \mathfrak{a},\mathfrak{b} \text{ are in } \eqref{dissip} \, (\text{see also Remark \ref{rmk: coeff}})\\
	\Pi_{2}(\varepsilon) & = \left\{ \pi \in \Pi \,:\,  \limsup\limits_{t\to +\infty} \int_{0}^{1} t^{\varepsilon} \wedge s^{-\varepsilon} \, \pi_{t}(\text{d}s)\, < \infty,\,  \int_{0}^{1}\int_{0}^{1} t^{\varepsilon}\wedge |s_{1} - s_{2}|^{-\varepsilon} \, \pi_{t}(\text{d}s_{1})\pi_{t}(\text{d}s_{2})\, <\infty\,\right\}.
\end{aligned}
\end{equation*}
These weight classes arise in numerous applications; see \cite[Example 2.1]{du2023empirical} or \ref{app: ex} in appendix.

\begin{theorem}\label{mainthm}
    Let $X_{\cdot}$ be the process in \eqref{CBO pi} with $\varpi\in\Pi_1(\varepsilon_1)$, and $Y_{\cdot}$ be the process in \eqref{CBOself pi}  with $\pi\in \Pi_1(\varepsilon_1)\cap \Pi_2(\varepsilon_2)$, where $\varepsilon_1,\varepsilon_2\in (0,1]$.  Let $\vartheta \in \Pi_{2}(\varepsilon_{2})$. 
    Assume  that $\mu^{*}_{\alpha} \in \mathscr{P}_{2,R}(\mathbb{R}^{d})$ is an invariant measure of \eqref{CBO kappa}. Further assume that $\lambda > 8\sigma^{2}$ and $0<\kappa\ll 1$. 
    Then there exists a constant $C$ independent of $t$ such that
	\begin{equation}\label{thmeq}
		\EE\left[ W_2^2(\mE_t^\vartheta[Y], \mu_\alpha^*) \right] +\EE  \left[ W_2^2 ( \mathcal{E}_{t}^{\vartheta}[X] ,\mu^*_{\alpha}) \right]  \leq C \, t^{-\varepsilon}\,, \quad \text{ where }  \; \varepsilon=\varepsilon_1\wedge \frac{1}{3(d+2)}\varepsilon_2.
	\end{equation}
\end{theorem}

\begin{remark}
    Example of weights $\pi,\varpi, \vartheta$ in Theorem \ref{mainthm} are: $\pi,\vartheta$ are Lebesgue and $\varpi = \delta_{1}$ is Dirac. Then $X_{\cdot}, Y_{\cdot}$ in the theorem become \eqref{CBO kappa} and \eqref{CBOself'} respectively,  $\mathcal{E}_{t}^{\vartheta}[\cdot]$ in \eqref{thmeq} becomes $\mathcal{E}_{t}[\cdot]$, and $\varepsilon_{1}<1-\mathfrak{b}/\mathfrak{a}$ and $\varepsilon_{2}<1$. This situation corresponds to the CBO with Personal Best in \cite[Equation (2.4)]{totzeck2020consensus}.
\end{remark}

To prove the theorem, let us first consider the following Markovian SDE  
\begin{align}\label{CBOh}
    \rd \hY_t&=-\lambda(\hY_t-\kappa\,\mathfrak{m}_\alpha(\mu_\alpha^*))\dt+\sigma\left( \frac{1}{\alpha} \mathds{I}_{d} + D(\hY_t-\kappa\,\mathfrak{m}_\alpha(\mu_\alpha^*))\right)\rd B_t.
\end{align} 
The following estimate will be useful and is a direct application of \cite[Proposition 3.1]{du2023empirical}. We recall $\delta>0$ is the one in \eqref{lemeq}, and is chosen to be equal to $1$ in Theorem \ref{mainthm}.
\begin{proposition}\label{prop}
    Let $\vartheta\in \Pi_2(\varepsilon_2)$ with $\varepsilon_2\in(0,1)$, and assume $\hY_{\cdot}$ satisfy \eqref{CBOh}. Then it holds
    \begin{equation} 
        \EE  \left[ W_2^2(\mE_t^\vartheta[\hY],\mu_\alpha^*) \right] \leq Ct^{-\gamma\,\varepsilon_2}\,, \quad \text{ where } \; 
        \gamma=\frac{\delta}{(d+2)(\delta+2)} 
    \end{equation}
\end{proposition}

\begin{proof}[Proof of Theorem \ref{mainthm}]
Let $\hY$ be as in \eqref{CBOh}. Given the result in Proposition \ref{prop}, it is then sufficient to estimate 	$\EE \left[W_2^2(\mE_t^\vartheta[Y],\mE_t^\vartheta[\hY])\right]$ and $\EE\left[W_2^2(\mE_t^\vartheta[X],\mE_t^\vartheta[\hY])\right]$. 

Let $Y$ and $\hY$ be solutions to  \eqref{CBOself pi} and \eqref{CBOh} respectively with the same deterministic initial data. Let us first estimate $\EE[|Y_t-\hY_t|^2]$ by using Ito's formula and Lemma \ref{lemsta}. Indeed it is easy to get 
\begin{align*}
    &\frac{\rd}{\dt}\EE[|Y_t-\hY_t|^2]\leq -\mathfrak{a} \, \EE[|Y_t-\hY_t|^2]+\mathfrak{b} \, \EE\left[ W_2^2(\mE_t^\pi[Y],\mu_\alpha^*)\right]\\
    & \quad \quad \leq   -\mathfrak{a} \, \EE[|Y_t-\hY_t|^2]+2\mathfrak{b} \, \EE\left[ W_2^2(\mE_t^\pi[Y],\mE_t^\pi[\hY])\right]+2\mathfrak{b} \, \EE\left[ W_2^2(\mE_t^\pi[\hY],\mu_\alpha^*)\right]\\
    & \quad \quad \leq  -\mathfrak{a} \, \EE[|Y_t-\hY_t|^2]+2\mathfrak{b}  \int_0^1\EE\left[|Y_{ts}-\hY_{ts}|^2\right]\pi_t(\rd s)+2\mathfrak{b} \, \EE[ W_2^2(\mE_t^\pi[\hY],\mu_\alpha^*)]\,,
\end{align*}
where $2\mathfrak{b}<\mathfrak{a}$. 
Since $\pi\in \Pi_2(\varepsilon_2)$, Proposition \ref{prop} (with $\pi$ instead of $\vartheta$) implies \; $\EE[ W_2^2(\mE_t^\pi[\hY],\mu_\alpha^*)]\leq Ct^{-\gamma \varepsilon_2}\,$.

Now let  $g(t):=\EE[|Y_t-\hY_t|^2]$, then it holds 
\begin{equation*}
    g'(t)\leq -\mathfrak{a} \, g(t)+2\mathfrak{b} \, \int_0^1g(ts)\pi_t(\ds)+C \, t^{-\gamma \varepsilon_2}.
\end{equation*}
By the Gronwall-type inequality  \cite[Lemma 4.1]{du2023empirical}, this leads to \; $g(t)=\EE[|Y_t-\hY_t|^2]\leq C t^{-(\varepsilon_1\wedge \gamma \varepsilon_2)}\,$, 
then 
\begin{align}\label{310}
\EE\left[W_2(\mE_t^\vartheta[Y],\mE_t^\vartheta[\hY])^2\right]\leq \int_0^1\EE\left[|Y_{ts}-\hY_{ts}|^2\right]\vartheta_t(\ds)\leq C\int_0^1(ts)^{-\varepsilon}\vartheta_t(\ds)\,,
\end{align}
where $\varepsilon=\varepsilon_1\wedge \gamma \varepsilon_2$. When $t$ is sufficiently large, one has $(ts)^{-\varepsilon}=1\wedge (ts)^{-\varepsilon}$ for all $s\in(0,1)$. Since $\vartheta\in \Pi_2(\varepsilon_2)$, the inequality \eqref{310} implies
\begin{align}\label{37}
    \EE\left[W_2(\mE_t^\vartheta[Y],\mE_t^\vartheta[\hY])^2\right]\leq C \, t^{-\varepsilon}\int_0^1t^{\varepsilon}\wedge s^{-\varepsilon}\vartheta_t(\ds)\leq C \, t^{-\varepsilon}\int_0^1t^{\varepsilon_2}\wedge s^{-\varepsilon_2}\vartheta_t(\ds)\leq C \, t^{-\varepsilon}\,.
\end{align}
The latter with the estimate of $\EE \left[W_2^2(\mE_t^\vartheta[\hY],\mu_\alpha^*)\right]$ in Proposition \ref{prop} yield  $\EE\left[ W_2^2(\mE_t^\vartheta[Y], \mu_\alpha^*) \right] \leq C \, t^{-\varepsilon}$. 

Similarly, one also has
\begin{align*}
    \frac{\rd}{\dt}\EE[|X_t-\hY_t|^2]
    \leq -\mathfrak{a} \, \EE[|X_t-\hY_t|^2]+2\mathfrak{b}\int_0^1\EE\left[|X_{ts}-\hY_{ts}|^2\right]\varpi_t(\rd s)+2\mathfrak{b} \, \EE\left[ W_2^2(\mc L_t^\varpi[\hY],\mu_\alpha^*)\right]\,.
\end{align*}
Then by \cite[Lemma 3.2]{du2023empirical} one has
\begin{align*}
    \EE\left[ W_2^2(\mc L_t^\varpi[\hY],\mu_\alpha^*)\right]\leq \int_0^1\EE\left[ W_2^2(\mc L_{ts}^\varpi[\hY],\mu_\alpha^*)\right]\varpi_t(\ds)\leq C\int_0^1e^{-\mathfrak{a} ts}\varpi_t(\ds)\leq Ct^{-\varepsilon_1}.
\end{align*}
This leads to  $\EE[|X_t-\hY_t|^2]\leq Ct^{-\varepsilon_1}$, which further implies  $\EE[W_2^2(\mE_t^\vartheta[X],\mE_t^\vartheta[\hY])]\leq Ct^{-\varepsilon}$ as in \eqref{310}-\eqref{37}.  Combining with Proposition \ref{prop}, we prove that $\EE[ W_2^2(\mE_t^\vartheta[X], \mu_\alpha^*) ] \leq Ct^{-\varepsilon}$. Therefore \eqref{thmeq} follows.
\end{proof}

The above theorem together with Proposition \ref{prop:existence of inv meas} imply that  $\mu^{*}_{\alpha}$ is the unique invariant measure of \eqref{CBO kappa}.
\begin{corollary}\label{cor:uniq}
	The process \eqref{CBO kappa} has a unique invariant measure.
\end{corollary}
\begin{proof}
Assume $\mu^*_\alpha$ and $\tilde \mu^*_\alpha$ are two invariant measures for \eqref{CBO kappa}. Then Theorem \ref{mainthm} concludes that \\
$W_2^2(\mu^*_\alpha,\tilde \mu^*_\alpha)\leq  	2\EE\left[ W_2^2(\mE_t^\vartheta[Y], \mu_\alpha^*) \right] +2	\EE\left[ W_2^2(\mE_t^\vartheta[Y],\tilde  \mu_\alpha^*) \right] \to 0 \mbox{ as }t\to \infty\,.$
\end{proof}

\begin{remark}
In the standard CBO model \eqref{CBO},  every Dirac measure is in fact an invariant measure. So there cannot be any guarantees on uniqueness of its long-time limit. It is only when $\alpha\to +\infty$ that one reaches  a global minimizer.  In our situation, both  \eqref{CBO kappa} and its self-interacting version \eqref{CBOself'} have the same unique invariant measure to which they converge in their long-time limit, for any fixed (finite) $\alpha$. 
\end{remark}

\begin{remark}
As a byproduct of our analysis, we showed that
the process \eqref{CBO kappa} is an example of a McKean-Vlasov SDE which has a unique invariant measure although it does not satisfy the usual dissipativity condition (also referred to as the confluence assumption). Indeed, as it can be observed from Lemma \ref{lemsta}, we do not require \eqref{dissip} to hold for all measures in $\mathscr{P}_{2}(\mathbb{R}^{d})\times \mathscr{P}_{2}(\mathbb{R}^{d})$, but rather only on the subset $\mathscr{P}_{2,R}(\RR^d)\times \mathscr{P}_{2}(\RR^d)$. We also obtain a polynomial rate of convergence in the sense of \eqref{thmeq}. 
In the former case, existence, uniqueness, and exponential convergence to the invariant measure hold true (see \cite{wang2018distribution, du2023empirical} and references therein). 
\end{remark}

Finally, one can also consider multi-self-interacting particles as it has been done in \cite[Theorem 2.3]{du2023empirical}. 

\begin{theorem}
Let $\pi \in \Pi_{1}(\varepsilon_{1})\cap \Pi_{2}(\varepsilon_{2})$ and $\vartheta \in \Pi(\varepsilon_{2})$ with $\varepsilon_{1},\varepsilon_{2}\in (0,1]$, and $Y^{i}_{\cdot}, i=1,\dots, N$ satisfy
\begin{equation}
\begin{aligned}
    & \rd Y^{i}_t = - \lambda ( Y^{i}_t -\kappa \,\mathfrak{m}_\alpha(\Upsilon^{N,\pi}_{t}[\mathbf{Y}]) \dt + \sigma \left(\frac{1}{\alpha}\mathds{I}_{d} + D(Y^{i}_t - \kappa\, \mathfrak{m}_\alpha(\Upsilon^{N,\pi}_{t}[\mathbf{Y}]) \right) \rd B^{i}_t\,,\\ 
    & \text{ where } \Upsilon^{N,\pi}_{t}[\mathbf{Y}] := \frac{1}{N}\sum\limits_{j=1}^{N} \mathcal{E}_{t}^{\pi}[Y^{j}],\quad \text{ and }\; \mathbf{Y}_{\cdot} := (Y^{1}_{\cdot},\dots,Y^{N}_{\cdot}).
\end{aligned}
\end{equation}
Here, $B^{i}_{\cdot}$ are $N$ independent Brownian motions, and $\mathbf{Y}_{0} = \mathbf{y}:=(y^{1},\dots,y^{N}) \in \mathbb{R}^{d\times N}$ is the initial condition. Then for each $\mathbf{y}$, there is a constant $C>0$ independent of $t$ and $N$ such that 
\begin{equation*}
    \mathbb{E}^{\mathbf{y}}\left[
    W_{2}^{2}(
     \Upsilon^{N,\vartheta}_{t}[\mathbf{Y}] , \mu^{*}_{\alpha}
    )\right]
    \leq C\, t^{-(\varepsilon_{1}\wedge \gamma \varepsilon_{2})}\, N^{-\gamma} + C\, t^{-(\varepsilon_{1}\wedge \varepsilon_{2})},\quad \text{ and } \gamma = \frac{1}{3(d+2)}.
\end{equation*}
\end{theorem} 

\begin{proof}
It is in the same line as the proof of Theorem \ref{mainthm}, where now we use \cite[Lemma 5.1]{du2023empirical} which is the multi-particle analogue of Proposition \ref{prop}. 
\end{proof}

\section*{Acknowledgments}

The authors wish to thank the two referees for their valuable comments which helped improve the paper.

\bibliographystyle{amsxport}
\bibliography{bibliography}

\appendix

\section{Supplementary material}\label{appendix}

\subsection{Proof of Lemma \ref{lemsta}}\label{app: lem}

We need to prove that $\exists\,\mathfrak{a}>2\mathfrak{b}\geq 0$ such that for all $x,y\in\RR^d$ and $(\mu,\nu)\in \mathscr{P}_{p,R}(\RR^d)\times \mathscr{P}_{p}(\RR^d)$ it holds that 
\begin{align}
    &  2 \la b(x,\mu)-	b(y,\nu),x-y\ra +\|\bm\sigma(x,\mu)-\bm\sigma(y,\nu)\|^2 
    \leq -\mathfrak{a}|x-y|^2+ \mathfrak{b} W_p^2(\mu,\nu)\notag
\end{align} 
Let us recall the dynamics
\begin{equation}
\rd X_t=-\lambda(X_t-\kappa \,\mathfrak{m}_{\alpha}(\mathcal{L}_{t}[X]))\dt+\sigma\left(\frac{1}{\alpha}\mathds{I}_{d} + D(X_t-\kappa\,\mathfrak{m}_{\alpha}(\mathcal{L}_{t}[X]))\right) \rd B_t\,
\end{equation}

We have
\begin{equation*}
\begin{aligned}
    & \langle b(x,\mu) - b(y,\nu),\,x-y \rangle = -\lambda |x-y|^{2} + \lambda \kappa \langle \mathfrak{m}_{\alpha}(\mu) - \mathfrak{m}_{\alpha}(\nu), \, x-y \rangle\\
    & \quad\quad \leq -\lambda |x-y|^{2} + \frac{\lambda \kappa}{2}|x-y|^{2} + \frac{\lambda\kappa}{2}|\mathfrak{m}_{\alpha}(\mu) - \mathfrak{m}_{\alpha}(\nu)|^{2}\\
    & \quad \quad \leq -\lambda\left(1-\frac{\kappa}{2}\right)|x-y|^{2} + \frac{\lambda\kappa}{2}L_{\mathfrak{m}}^{2}W_{p}^{2}(\mu,\nu)
\end{aligned}
\end{equation*}
where in the second inequality we have used Lemma \ref{lem: useful estimates}. 
On the other hand, we have
\begin{equation*}
\begin{aligned}
    & \|\bm\sigma(x,\mu)-\bm\sigma(y,\nu)\|^2 = \sigma^{2}\left|(x - \kappa\, \mathfrak{m}_{\alpha}(\mu)) - (y - \kappa\, \mathfrak{m}_{\alpha}(\nu)) \right|^{2}\\
    & \quad \quad \leq 2\sigma^{2}|x-y|^{2} + 2\sigma^{2}\kappa^{2}|\mathfrak{m}_{\alpha}(\mu) - \mathfrak{m}_{\alpha}(\nu)|^{2}\\
    & \quad \quad \leq 2\sigma^{2}|x-y|^{2} + 2\sigma^{2}\kappa^{2} L_{\mathfrak{m}}^{2}W_{p}^{2}(\mu,\nu).
\end{aligned}
\end{equation*}
Therefore, we have
\begin{equation*}
\begin{aligned}
    & 2 \la b(x,\mu)-	b(y,\nu),x-y\ra +\|\bm\sigma(x,\mu)-\bm\sigma(y,\nu)\|^2\\
    &\quad \quad  \leq -\lambda\left(2-\kappa\right)|x-y|^{2} + \lambda\kappa L_{\mathfrak{m}}^{2}W_{p}^{2}(\mu,\nu)\\
    &\quad \quad \quad \quad \quad + 2\sigma^{2}|x-y|^{2} + 2\sigma^{2}\kappa^{2} L_{\mathfrak{m}}^{2}W_{p}^{2}(\mu,\nu)\\
    &\quad \quad = -(2\lambda-\lambda \kappa - 2\sigma^{2})|x-y|^{2} + (\lambda + 2\sigma^{2}\kappa)\kappa L_{\mathfrak{m}}^{2}\,W_{p}^{2}(\mu,\nu)
\end{aligned}
\end{equation*}
This suggests that $\mathfrak{a} = 2\lambda-\lambda \kappa - 2\sigma^{2}$ and $\mathfrak{b} = (\lambda + 2\sigma^{2}\kappa)\kappa L_{\mathfrak{m}}^{2}$. We need to check that $\mathfrak{a} >2 \mathfrak{b} \geq 0$. We always have $(\lambda + 2\sigma^{2}\kappa)\kappa L_{\mathfrak{m}}^{2}\geq 0$. So we check that  $2\lambda-\lambda \kappa - 2\sigma^{2} > (\lambda + 2\sigma^{2}\kappa)\kappa L_{\mathfrak{m}}^{2}$. This can be satisfied provided we have $\lambda > \sigma^{2}$ and we choose $0<\kappa \ll 1$.

Next, we need to prove that there exist $\mathfrak{c}>0$, $\delta>0$ and $K\geq 0$ such that for all $x\in\RR^d, \nu\in \mathscr{P}_2(\RR^d)$ it holds
\begin{align}
    2\la b(x,\nu),x\ra +(1+\delta)\|\bm\sigma(x,\nu)\|^2
    \leq -\mathfrak{c}|x|^2+ K[1 + \delta + \nu(|\cdot|^{2})] \,.\notag
\end{align}

We have
\begin{equation*}
\begin{aligned}
    \langle b(x,\nu),x\rangle & = -\lambda \langle x - \kappa \,\mathfrak{m}_{\alpha}(\nu),\, x \rangle = -\lambda |x|^{2} + \lambda \kappa \,\langle  \mathfrak{m}_{\alpha}(\nu),\,x\rangle\\
    & \leq -\left(\lambda - \frac{\lambda \kappa}{2}\right) |x|^{2} + \frac{\lambda \kappa}{2} |\mathfrak{m}_{\alpha}(\nu)|^{2}\\
    & \leq  -\left(\lambda - \frac{\lambda \kappa}{2}\right) |x|^{2} + \frac{\lambda \kappa}{2}C_{1}^{2} \, \nu(|\cdot|^{2})
\end{aligned}
\end{equation*}
where in the second inequality we have used Lemma \ref{lem: useful estimates}. 
On the other hand, we have 
\begin{equation*}
\begin{aligned}
    \|\bm\sigma(x,\nu)\|^{2} & = \sigma^{2}\sum_{k=1}^{d}\left(\frac{1}{\alpha}+(x - \kappa\, \mathfrak{m}_{\alpha}(\nu))_k \right)^{2} \\
    & \leq 2\sigma^{2}\left( \frac{d}{\alpha^{2}} + | x - \kappa \,\mathfrak{m}_{\alpha}(\nu)|^{2} \right)\\
    & = 2\sigma^{2}\big( |x|^{2} - 2\kappa\, \langle \mathfrak{m}_{\alpha}(\nu),\,x \rangle + \kappa^{2}|\mathfrak{m}_{\alpha}(\nu)|^{2} \big) + \frac{2\sigma^{2}d}{\alpha^{2}}\\
    & \leq 2\sigma^{2}\big( |x|^{2} + \kappa |x|^{2} + \kappa |\mathfrak{m}_{\alpha}(\nu)|^{2}   + \kappa^{2}|\mathfrak{m}_{\alpha}(\nu)|^{2} \big) + \frac{2\sigma^{2} d }{\alpha^{2}}\\
    & = 2\sigma^{2}(1+\kappa)|x|^{2} + 2\sigma^{2}(1+\kappa)\kappa |\mathfrak{m}_{\alpha}(\nu)|^{2} + \frac{2\sigma^{2} d }{\alpha^{2}}\\
    & \leq 2\sigma^{2}(1+\kappa)|x|^{2} + 2\sigma^{2}(1+\kappa)\kappa C_{1}^{2}\, \nu(|\cdot|^{2}) + \frac{2\sigma^{2} d }{\alpha^{2}}.
\end{aligned}
\end{equation*}
Therefore we have
\begin{equation*}
\begin{aligned}
    & 2 \langle b(x,\nu),x\rangle + (1+\delta )\|\bm\sigma(x,\nu)\|^{2} \\
    & \quad \quad \leq -(2\lambda - \lambda \kappa ) |x|^{2} + \lambda \kappa \,C_{1}^{2} \, \nu(|\cdot|^{2})  + (1+\delta )\frac{2\sigma^{2} d }{\alpha^{2}}\\
    & \quad \quad \quad \quad + 2\sigma^{2}(1+\delta )(1+\kappa)|x|^{2} + 2\sigma^{2}(1+\delta )(1+\kappa)\kappa C_{1}^{2}\, \nu(|\cdot|^{2})\\
    & \quad \quad = - \big[(2\lambda - \lambda \kappa ) - 2\sigma^{2}(1+\delta)(1+\kappa)\big]\, |x|^{2} \\
    & \quad \quad \quad \quad + \big[\lambda \kappa \,C_{1}^{2} + 2\sigma^{2}(1+\delta)(1+\kappa)\kappa C_{1}^{2}\big]\, \nu(|\cdot|^{2})+ (1+\delta )\frac{2\sigma^{2} d }{\alpha^{2}}.
\end{aligned}
\end{equation*}
This suggests that we can take 
\begin{align*}
     \mathfrak{c} = (2\lambda - \lambda \kappa ) - 2\sigma^{2}(1+\delta)(1+\kappa) 
    \text{ and }  K= \max\left\{ 2\sigma^{2}d/\alpha^{2} ,\; \big[\lambda \kappa \,C_{1}^{2} + 2\sigma^{2}(1+\delta)(1+\kappa)\kappa C_{1}^{2}\big] \right\}.
\end{align*} 
It remains to check if $\mathfrak{c}>0$. A sufficient condition for the latter to hold is to have $\lambda  >  4\sigma^{2}(1+\delta)$, which can be satisfied for example by choosing $\delta = 1$ and $\lambda > 8 \sigma^{2}$. Note that this condition complies with the condition ($\lambda > \sigma^{2}$) found earlier. Therefore we have $K\geq 0$ and $\delta>0$ such that
\begin{equation*}
\begin{aligned}
    & 2 \langle b(x,\nu),x\rangle + (1+\delta )\|\bm\sigma(x,\nu)\|^{2} \leq -\mathfrak{c} |x|^{2} + K[1+\delta + \nu(|\cdot|^{2})], \quad \text{ and } \mathfrak{c}>0.
\end{aligned}
\end{equation*}
Simplified sufficient conditions can be obtained as follows:
\begin{equation*}
\begin{aligned}
    \mathfrak{c}>0 \Leftrightarrow\;& (2\lambda - \lambda \kappa ) - 2\sigma^{2}(1+\delta)(1+\kappa) >0\\
    \Leftrightarrow \; & (2\lambda - \lambda \kappa ) > 2\sigma^{2}(1+\delta)(1+\kappa)\\
    \Leftrightarrow \; & \lambda (2-\kappa) >   2\sigma^{2}(1+\delta)(1+\kappa)\\
    k\in (0,1) \Rightarrow & \lambda (2-\kappa) > \lambda  \\
    \text{it is sufficient to have } \; & \lambda  >  4\sigma^{2}(1+\delta)  >  2\sigma^{2}(1+\delta)(1+\kappa)\\
    \text{it is sufficient to have } \; & \lambda  > 8\sigma^{2} \text{ when } \delta =1. 
\end{aligned}
\end{equation*}

\subsection{Proof of Proposition 3.4}\label{app: prop}

Let us check that an invariant measure for (2.2) 
exists and is unique. Moreover, it is in $\mathscr{P}_{p,R}(\mathbb{R}^{d})$ for some sufficiently large fixed $R>0$. To do so, we shall apply the result of S.-Q. Zhang \cite[Theorem 2.2]{zhang2023existence}.

\textbullet\; Assumption \cite[(H1)]{zhang2023existence}\\
Let us choose $r_{1} = r_{2} = 1$ and $r_{3} = 1+r_{2}= 2$. Then we have, for all $\nu \in \mathscr{P}_{2}(\mathbb{R}^{d})$,
\begin{equation*}
\begin{aligned}
    & 2 \langle b(x,\nu),\nu\rangle + \|\bm\sigma(x,\nu)\|^{2} \\
    & \quad \quad \leq -(2\lambda - \lambda \kappa ) |x|^{2} + \lambda \kappa \,C_{1}^{2} \, \nu(|\cdot|^{2}) \\
    & \quad \quad \quad \quad +  2\sigma^{2}(1+\kappa)|x|^{2} + 2\sigma^{2}(1+\kappa)\kappa C_{1}^{2}\, \nu(|\cdot|^{2}) + \frac{2\sigma^{2} d }{\alpha^{2}}\\
    & \quad \quad = - \big[(2\lambda - \lambda \kappa ) - 2\sigma^{2}(1+\kappa)\big]\, |x|^{2} + \frac{2\sigma^{2} d }{\alpha^{2}}\\
    & \quad \quad \quad \quad + \big[\lambda \kappa \,C_{1}^{2} + 2\sigma^{2}(1+\kappa)\kappa C_{1}^{2}\big]\, \nu(|\cdot|^{2})
\end{aligned}
\end{equation*}
and it can be verified that $\tilde{C}_{1} := (2\lambda - \lambda \kappa ) - \sigma^{2}(1+\kappa^{2})$ is a positive constant as long as $\lambda$ is chosen large enough comparing to $\sigma^{2}$, and for $0<\kappa \ll 1$. We also need to check that $\tilde{C}_{1} > \tilde{C}_{3}$ where $\tilde{C}_{3} := \lambda \kappa \,C_{1}^{2} + \sigma^{2}(1+\kappa)\kappa C_{1}^{2}$. This can also be verified since $\kappa$ can be chosen small enough. Therefore, denoting by $\tilde{C}_{2}=2\sigma^{2}d/\alpha^{2}$, we have
\begin{equation*}
\begin{aligned}
    & 2 \langle b(x,\nu),\nu\rangle + \|\bm\sigma(x,\nu)\|^{2}  \leq - \tilde{C}_{1}|x|^{2} + \tilde{C}_{2} + \tilde{C}_{3} \nu(|\cdot|^{2}), \quad \text{ and } \tilde{C}_{1}>\tilde{C}_{3}
\end{aligned}
\end{equation*}

\textbullet\; Assumption \cite[(H2.i)]{zhang2023existence}\\
For $\nu$ fixed in $\mathscr{P}_{2}(\mathbb{R}^{d})$, we need to check that the drift and diffusion terms are locally Lipschitz, that is: \\
For every $n\in \mathbb{N}$ and $\nu \in \mathscr{P}_{2}(\mathbb{R}^{d})$, there exists $K_{n}>0$ such that for all $|x|\vee |y|\leq n$ we have
\begin{equation*}
    |b(x,\nu) - b(y,\nu)| + \|\bm\sigma(x,\nu) - \bm\sigma(y,\nu)\| \leq K_{n}|x-y|.
\end{equation*}
This is easily verified from the definition of $b$ and $\bm\sigma$ (in fact, $K_{n}=\lambda + \sigma$). 

\textbullet\; Assumption \cite[(H2.ii)]{zhang2023existence}\\
We need to check that the drift has a polynomial growth. More precisely we want to check that: there exists a locally bounded function $\mathfrak{h}:[0,+\infty)\to [0,+\infty)$ such that
\begin{equation*}
    |b(x,\nu)| \leq \mathfrak{h}(\nu(|\cdot|^{2}))\, (1+|x|),\quad x\in \mathbb{R}^{d},\; \nu \in \mathscr{P}_{2}(\mathbb{R}^{d}).
\end{equation*}
This holds true, noting that
\begin{equation*}
\begin{aligned}
    |b(x,\nu)| & = \lambda|x- \kappa \,\mathfrak{m}_{\alpha}(\nu)|\leq \lambda\big(|x| + \kappa |\mathfrak{m}_{\alpha}(\mu)|\big)\\
    & \leq \lambda|x| + \lambda\kappa C_{1}\nu(|\cdot|^{2})^{\frac{1}{2}},\quad \text{using Lemma 3.1}\\
    & \leq \lambda(1+|x|) + \lambda\kappa C_{1}(1+|x|)\nu(|\cdot|^{2})^{\frac{1}{2}}\\
    & \leq \mathfrak{h}(\nu(|\cdot|^{2}))\, \big(1 + |x| \big)
\end{aligned}
\end{equation*}
where we have set $\mathfrak{h}(\xi) = \lambda + \lambda\kappa C_{1}\,\xi^{\frac{1}{2}}$ for every $\xi\in [0,+\infty)$.

\textbullet\; Assumption \cite[(H3)]{zhang2023existence}\\
We need the drift and diffusion coefficients to be continued on $\mathscr{P}_{2,R}(\mathbb{R}^{d})$ equipped with the Wasserstein metric. This is guaranteed thanks to Lemma \ref{lem: useful estimates}.

\textbullet\; Assumption \cite[(H4)]{zhang2023existence}\\
We need the diffusion matrix to be non-degenerate on $\mathbb{R}^{d}\times \mathscr{P}_{2}(\mathbb{R}^{d})$ , that is
\begin{equation*}
    \bm\sigma(x,\nu)\bm\sigma^{*}(x,\nu)>0,\quad x\in \mathbb{R}^{d},\; \nu \in \mathscr{P}_{2}(\mathbb{R}^{d}).
\end{equation*}
This is guaranteed thanks to the additional term $\frac{1}{\alpha}\mathds{I}_{d}$ in the definition of $\bm{\sigma}$, noting that $D(\cdot)$ is a non-negative matrix.

\textbullet\; Conclusion:\\
The latter assumptions being satisfied, we can then apply  \cite[Theorem 2.2]{zhang2023existence}, which guarantees that the distribution-dependent SDE (DDSDE)
\begin{equation}
\rd X_t=-\lambda(X_t-\kappa\, \mathfrak{m}_{\alpha}(\mathcal{L}_{t}[X]))\dt+\sigma\left(\frac{1}{\alpha}\mathds{I}_{d} + D(X_t-\kappa\,\mathfrak{m}_{\alpha}(\mathcal{L}_{t}[X]))\right) \rd B_t\,
\end{equation}
has a stationary distribution. In fact, from the proof in \cite{zhang2023existence} (see the last line in page 8, and  the lines between equation (2.22) and equation (2.23) in page 10), it appears that there exists $R_{0}>0$ depending only on the constants of the problem ($\lambda, \kappa, \sigma, \alpha$, and $C_{1}$ from Lemma \ref{lem: useful estimates}) 
such that the stationary distribution exists in $\mathscr{P}_{2,R}(\mathbb{R}^{d})$, and $R\geq R_{0}$ is determined by the initial distribution.

\subsection{Additional computations}\label{app: comp}

In order to better see how  $\eta_\alpha^*(\dx):=\frac{\omega_\alpha^f(x)\mu_\alpha^*(\dx)}{\int_{ \RR^d }\omega_\alpha^f(x)\mu_\alpha^*(\dx)}$ approximates the Dirac distribution $\delta_{x^*}$ for large $\alpha\gg 1$, we proceed with the following computations. Recall Laplace's principle
\begin{equation*}
	\lim\limits_{\alpha\to\infty}\left(-\frac{1}{\alpha}\log\left(\int_{ \RR^d }\omega_\alpha^f(x)\mu_\alpha^*(\dx)\right)\right)= f(x^*)\,.
\end{equation*}
The latter is equivalent to 
\begin{equation*}
    \lim\limits_{\alpha\to\infty}\left(\int_{ \RR^d }\omega_\alpha^f(x)\mu_\alpha^*(\dx)\right)^{\frac{1}{\alpha}}= e^{-f(x^*)}\,,
\end{equation*}
which also means
\begin{equation*}
\begin{aligned}
    \lim\limits_{\alpha\to\infty} \frac{e^{-f(x^*)}}{\left(\int_{ \RR^d }\omega_\alpha^f(x)\mu_\alpha^*(\dx)\right)^{\frac{1}{\alpha}}} = 1,
\end{aligned}
\end{equation*}
or equivalently (by rising to power $\alpha$)
\begin{equation*}
\begin{aligned}
    \lim\limits_{\alpha\to\infty} \frac{e^{-\alpha f(x^*)}}{\int_{ \RR^d }\omega_\alpha^f(x)\mu_\alpha^*(\dx)} = 1.
\end{aligned}
\end{equation*}
It suffices now to note that the left-hand side is the integration of the indicator function $\textbf{I}_{\{x^*\}}(\cdot)$ supported on the global minimizer $x^{*}$ with respect to $\eta_\alpha^*(\dx)$. Indeed, 
we have
\begin{equation*}
    \frac{e^{-\alpha f(x^*)}}{\int_{ \RR^d }\omega_\alpha^f(x)\mu_\alpha^*(\dx)}  = \la \eta_\alpha^*(\dx),\textbf{I}_{\{x^*\}}\ra
\end{equation*}
where $\langle\cdot\,,\,\cdot\rangle$ denotes the integration of a function with respect to a measure. Therefore, when $\alpha$ is large enough, one expects $\la \eta_\alpha^*(\dx),\textbf{I}_{\{x^*\}}\ra \approx 1$. 
Thus $\eta_\alpha^*$ approximates the Dirac distribution $\delta_{x^*}$ for large $\alpha\gg 1$. 

\subsection{Example of weight measures}\label{app: ex}

For the sake of self-containedness of the paper,  we borrow from \cite[Example 2.1]{du2023empirical} an example of a weight measure that arises in applications. 
In the dynamics \eqref{CBOself pi}, one can choose the weight measure $\hat{\pi}$ as defined in \cite{du2023empirical}, that is
\begin{equation}\label{pi hat example}
    \hat{\pi}_{t}(\cdot) = \frac{1}{n_{t}} \sum\limits_{k=0}^{n_{t}-1} \delta_{\frac{k\tau - \theta}{t}\vee 0}(\cdot)\quad \text{ with } n_{t} = \left\lceil \frac{t}{\tau} \right\rceil
\end{equation}
where $\tau >0$ is the \underline{sampling period} and $\theta\geq 0$ is the \underline{delay}. 
Let us denote by $\hat{Y}$ the resulting process that is
\begin{equation*}
	\rd \hat{Y}_t = - \lambda ( \hat{Y}_t -\kappa \,\mathfrak{m}_\alpha(\mE_{t}^{\hat{\pi}}[\hat{Y}])) \dt + \sigma \left(\frac{1}{\alpha}\mathds{I}_{d} + D(\hat{Y}_t - \kappa\, \mathfrak{m}_\alpha(\mE_{t}^{\hat{\pi}}[\hat{Y}])) \right) \rd B_t\,, \;\hat{Y}_{0}=x\in \mathbb{R}^{d}.
\end{equation*}
The latter measure $\hat{\pi} = (\hat{\pi}_{t})_{t\geq 0}$ belongs to $\Pi_{1}(1-\mathfrak{b}/\mathfrak{a})$ and $\Pi_{2}(\varepsilon)$ for any $\varepsilon<1$.

\hfill

A \underline{\textbf{practical example}} could be:\; $\tau = 1/T$\; where $T$ is the time horizon, and \; $\theta = 2\,\tau$. 

\hfill\\
Then, at each given time $t\in (0,T)$, we shall introduce $m\in \mathbb{N}$ (depending on $t$) such that
\begin{equation*}
    m\tau \leq t < (m+1)\tau \quad \Leftrightarrow \quad \frac{m}{T} \leq t < \frac{m+1}{T}.
\end{equation*}
In the notation of \eqref{pi hat example}, we have $\quad m= \left\lfloor \frac{t}{\tau} \right\rfloor \quad \Leftrightarrow \quad m \leq \frac{t}{\tau} < m+1 \quad \Leftrightarrow \quad m=n_{t}-1.$ \\
In this example, the measure \eqref{pi hat example} becomes
\begin{equation}\label{pi hat example 2}
\begin{aligned}
    \hat{\pi}_{t}(\cdot)
    & = \frac{1}{n_{t}} \sum\limits_{k=0}^{n_{t}-1} \delta_{\frac{k\tau - \theta}{t}\vee 0}(\cdot)\quad (\text{with } n_{t} = \left\lceil \frac{t}{\tau} \right\rceil = m+1, \quad \text{ and } \quad \theta=2\,\tau)\\
    & = \frac{1}{m+1} \sum\limits_{k=0}^{m} \delta_{\frac{(k-2)\tau}{t}\vee 0}(\cdot) 
    \; \; = \begin{cases}
        \; \delta_0,\quad m=0,1\\
        \; \frac{1}{m+1} \left(2\,\delta_{0}(\cdot) + \sum\limits_{k=2}^{m} \delta_{\frac{(k-2)\tau}{t}}(\cdot)\right),\quad m\geq 2
    \end{cases}\\
    & = \begin{cases}
       \; \delta_0,\quad m=0,1\\
        \; \frac{1}{m+1} \left(2\,\delta_{0}(\cdot) +\sum\limits_{k=0}^{m-2} \delta_{\textcolor{red}{k\frac{\tau}{t}}}(\cdot)\right),\quad m\geq 2
    \end{cases}
\end{aligned}
\end{equation}
Hence, one obtains when $m\geq 2$
\begin{equation*}
\begin{aligned}
    \mE_t^{\hat{\pi}}[Y] 
    & := \int_0^1 \delta_{Y_{\textcolor{red}{s}t}}\,\hat{\pi}_t(\text{d}\textcolor{red}{s}) = \frac{1}{m+1}\left(2\delta_{Y_{0}} +  \sum\limits_{k=0}^{m-2}\delta_{Y_{\textcolor{red}{k\frac{\tau}{t}}t}}\right) \quad (\text{here } s \leftarrow \textcolor{red}{k\frac{\tau}{t}})\\
    & = \frac{1}{m+1}\left(2\delta_{Y_{0}} + \sum\limits_{k=0}^{m-2}\delta_{Y_{\frac{k}{T}}}\right) \quad (\text{here } \tau = \frac{1}{T}).
\end{aligned}
\end{equation*}
Observe that $\mE_t^{\hat{\pi}}[Y]$  only depends on the observation of $Y_{\cdot}$ at the time steps $\frac{0}{T}, \frac{1}{T}, \frac{2}{T}, \dots, \frac{m-2}{T}.$ The last two time steps $\frac{m-1}{T}$ and $\frac{m}{T}$ are not included due to the chosen delay $\theta=2\,\tau$.  Moreover, we have 
\begin{equation*}
\begin{aligned}
    \mE_t^{\hat{\pi}}[Y]=\frac{1}{m+1}\left(2\delta_{Y_{0}} + \sum\limits_{k=0}^{m-2}\delta_{Y_{\frac{k}{T}}}\right) 
    & = \frac{1}{m+1}\left(2\delta_{Y_{0}} + \sum\limits_{k=0}^{m-2}\delta_{Y_{k\tau}}\right)\\
    & = \frac{1}{m+1}\left(2\delta_{Y_{0}} + \sum\limits_{k=0}^{m-2}\delta_{Y_{\frac{k}{m}\textcolor{red}{m\tau}}}\right)\\
    & = \int_0^1 \delta_{Y_{s\textcolor{red}{m\tau}}}\,\hat{\pi}_{\textcolor{red}{m\tau}}(\text{d}s)=\mathcal{E}^{\hat{\pi}}_{\textcolor{red}{m\tau}} [Y]
\end{aligned}
\end{equation*}
where $\hat{\pi}_{m\tau}(\cdot)$ is obtained using \eqref{pi hat example 2} with $\textcolor{red}{m\tau}$ instead of $t$ therein, that is
\begin{equation*}
\begin{aligned}
    \hat{\pi}_{\textcolor{red}{m\tau}}(\cdot)
    & = \frac{1}{m+1} \left(2\,\delta_{0}(\cdot) +\sum\limits_{k=0}^{m-2} \delta_{\frac{k}{m}}(\cdot)\right).
\end{aligned}
\end{equation*}
Finally in this example, the consensus point becomes
\begin{equation*}
    \mathfrak{m}_\alpha(\mE_t^{\hat{\pi}}[\hat{Y}]) = \frac{2\,\hat{Y}_{0}\exp(-\alpha f(\hat{Y}_{0})) + \sum\limits_{k=0}^{m-2} \hat{Y}_{\frac{k}{T}}\exp(-\alpha f(\hat{Y}_{\frac{k}{T}}))}{
    2\, \exp(-\alpha f(\hat{Y}_{0})) + \sum\limits_{\ell=0}^{m-2} \exp(-\alpha f(\hat{Y}_{\frac{\ell}{T}}))}.
\end{equation*}
When $m=1$, i.e. $\frac{1}{T} \leq t < \frac{2}{T}$,  we have $\mE_t^{\hat{\pi}}[\hat{Y}]= \delta_{\hat{Y}_{0}}$, and the consensus point is
\begin{equation*}
    \mathfrak{m}_\alpha(\mE_t^{\hat{\pi}}[\hat{Y}]) = \frac{\hat{Y}_{0}\exp(-\alpha f(\hat{Y}_{0}))}{\exp(-\alpha f(\hat{Y}_{0}))} = \hat{Y}_{0}.
\end{equation*}
Similarly, when $m=0$, i.e. $0 \leq t < \frac{1}{T}$, we also have $\mE_t^{\hat{\pi}}[\hat{Y}] = \delta_{\hat{Y}_{0}}$ and $\mathfrak{m}_\alpha(\mE_t^{\hat{\pi}}[\hat{Y}]) = \hat{Y}_{0}$.

\hfill

\hfill

\hfill

\end{document}